\documentclass[11pt,a4paper,reqno]{article}
\usepackage[utf8]{inputenc}
\usepackage[T1]{fontenc}
\usepackage[english]{babel}
\usepackage{hyperref}
\usepackage{graphicx}
\usepackage{amsfonts,amssymb,amsbsy,latexsym,a4wide,xspace}
\usepackage{amsmath,delarray,enumerate,calc,amsthm}
\usepackage{bbm}
\usepackage{txfonts}
\usepackage{paralist}

\usepackage{authblk}
\usepackage{mathabx}  

\usepackage{tikz-cd}
\usetikzlibrary{matrix}

\usepackage{url}

\newcommand\blfootnote[1]{%
  \begingroup
  \renewcommand\thefootnote{}\footnote{#1}%
  \addtocounter{footnote}{-1}%
  \endgroup
}

\newcommand{\cA}{{\mathcal A}}

\newcommand{\cB}{{\mathcal B}}

\newcommand{\cF}{{\mathcal F}}

\newcommand{\cM}{{\mathcal M}}

\newcommand{\N}{{\mathbbm N}}

\newcommand{\Z}{{\mathbbm Z}}

\newcommand{\lcm}{\operatorname{lcm}}

\newcommand{\card}{\operatorname{card}}

\newcommand{\dl}{\underline{d}}
\newcommand{\du}{\overline{d}}

\newcommand{\Spec}{\operatorname{Spec}}

\newtheorem {lemma}{Lemma}
\newtheorem{theorem}{Theorem}
\newtheorem {bemerkung}{Remark}
\newtheorem{proposition}{Proposition}
\newtheorem{beispiel}{Example}
\newtheorem{frage}{Question}
\newtheorem{vermutung}{Conjecture}

\newenvironment{remark} {\begin{bemerkung} \normalfont }{\end{bemerkung}}

\begin{document}

\title{Irregular $\cB$-free Toeplitz sequences via\\ Besicovitch's construction of sets of multiples without density}
\author{Gerhard Keller
\thanks{Research partly supported by the Polish National Science Center UMO-2019/33/B/ST1/00364 }
}
\affil{Department of Mathematics, University of Erlangen-N\"urnberg, \\
 Cauerstr. 11, 91058 Erlangen, Germany
 \par Email: keller@math.fau.de}

\maketitle
\begin{abstract}
Modifying Besicovitch's construction of a set $\cB$ of positive integers whose set of multiples $\cM_\cB$ has no asymptotic density, we provide examples of such sets $\cB$ for which $\eta:=1_{\Z\setminus\cM_\cB}\in\{0,1\}^\Z$ is a Toeplitz sequence. Moreover our construction produces examples, for which $\eta$ is not only quasi-generic for the Mirsky measure (which has discrete dynamical spectrum), but also for some measure of positive entropy. On the other hand, modifying slightly an example from Kasjan, Keller, and Lema\'nczyk, we construct a set $\cB$ for which $\eta$ is an irregular Toeplitz sequence but for which the orbit closure of $\eta$ in $\{0,1\}^\Z$ is uniquely ergodic.
\end{abstract}

\section{Introduction}
For subsets $\cB\subseteq\N\setminus\{1\}$  denote
by $\cM_\cB:=\bigcup_{b\in\cB}b\Z$ the set of all multiples of $\cB$. The density 
$d(\cM_\cB):=\lim_{N\to\infty}N^{-1}\card(\cB\cap[1:N])$ exists in many cases, 
in particular if $\cB$ is \emph{thin}, i.e.~if $\sum_{b\in\cB}1/b<\infty$,
but Besicovitch \cite{Besicovitch1935} provided examples where it does not exist. A bit later, Davenport and Erd\H{o}s \cite{DavenportErdos1937} proved that the logarithmic density of $\cM_\cB$ always exists and coincides with the lower density $\dl(\cM_\cB)$. For more background on this see \cite[Sec.~2.5]{BKKL2015}.

Recurrence properties of the sequence $\eta=1_{\cF_\cB}:=1_{\Z\setminus\cM_\cB}\in\{0,1\}^\Z$ can be studied using dynamical systems theory, see in particular \cite{BKKL2015}. To that end denote by $S$ the left shift on $\{0,1\}^\Z$ and restrict this homeomorphism to the closure $X_\eta$ of $\{S^n\eta:n\in\Z\}$ in $\{0,1\}^\Z$. There is a distinguished invariant measure $\mu_\eta$ on $X_\eta$ (its Mirsky measure) for which $\eta$ is quasi-generic. 
Indeed, $N^{-1}\sum_{n=1}^N\delta_{S^n\eta}$ converges weakly to $\mu_\eta$ along each subsequence $(N_i)_i$ along which the lower density $\dl(\cM_\cB)$ is attained.
Hence $\eta$ is generic for $\mu_\eta$ if and only if $d(\cM_\cB)$ exists.
It follows that for the sets $\cB$ constructed by Besicovitch \cite{Besicovitch1935}, the point $\eta$ is quasi-generic for at least one further invariant measure on $X_\eta$.

In \cite{BKKL2015} and \cite{KKL2016} properties of the topological and measure theoretic dynamical systems $(X_\eta,S)$ and $(X_\eta,S,\mu_\eta)$ were characterized in (elementary) number theoretic terms.
Combining some of these results, it turns out that Besicovitch's examples always lead to proximal systems $(X_\eta,S)$, namely to systems where each closed invariant subset contains the fixed point $0^\Z$, and that such systems have positive topological entropy and host a huge collection of ergodic invariant measures, among them a unique measure of maximal entropy, see Remark~\ref{remark:Besicovitch-proximal} in Section~\ref{sec:like-Besicovitch}.
But there are also plenty of sets $\cB$ for which the system $(X_\eta,S)$ itself is minimal. In these cases, $\eta$ is a Toeplitz sequence (see Remark~\ref{remark:toeplitz}), and it is hitherto unknown whether
there are examples of this type where $X_\eta$ can host more invariant measures than just the Mirsky measure.\footnote{For general Toeplitz shifts it is known that this is possible \cite{Williams1984}.}

Following Hall \cite{Hall1996}, we call $\cB$ a \emph{Besicovitch set}, if the density $d(\cM_\cB)$ exists.
In this note we modify Besicovitch's construction and prove the the following result:
\footnote{The question answered by Theorem~\ref{theo:1} arose when working on \cite{KKL2016} with Stanis\l{}aw Kasjan and Mariusz Lema\'n{}czyk. It was explicitly formulated in \cite{Dymek2017}.}

\begin{theorem}\label{theo:1}
There are sets $\cB\subseteq\N\setminus\{1\}$ with the following properties:
\begin{compactenum}[i)]
\item The sequence $\eta=1_{\cF_\cB}$ is an irregular Toeplitz sequence.
\item The set of shift invariant measures on $X_\eta$ contains at least one measure of positive entropy.
\item Depending on the details of the construction one can make sure that
\begin{compactenum}[(a)]
\item $\cB$ is not a Besicovitch set and $\eta$ is quasi-generic for some measure of positive entropy, or
\item $\cB$ is a Besicovitch set, so $\eta$ is generic for the Mirsky measure, but there is also some measure of positive entropy.
\end{compactenum}
\end{compactenum}
\end{theorem}
\begin{remark}\label{remark:toeplitz}
Recall (e.g. from \cite[Sec.~2]{Downarowicz2015})  that a sequence $\omega\in\{0,1\}^\Z$ is a \emph{Toeplitz sequence}, if for each $n\in\Z$ there exists a positive integer $p$ such that $\omega_n=\omega_{n+kp}$ for all $k\in\Z$. Given $p\in\N$, denote ${\sf Per}_p(\omega)=\{n\in\Z: \omega_n=\omega_{n+kp}\text{ for all }k\in\Z\}$ and ${\sf Aper}(\omega)=\Z\setminus\bigcup_{p\in\N}{\sf Per}_p(\omega)$. With this notation, $\omega$ is a Toeplitz sequence if and only if ${\sf Aper}(\omega)=\emptyset$. A Toeplitz sequence $\omega$ is \emph{regular}, if $\sup_{p\in\N}d({\sf Per}_p(\omega))=1$ \cite[Thm.~2.8]{Downarowicz2015}, otherwise it is \emph{irregular}. The reader may have in mind that $\bigcup_{p=1}^N{\sf Per}_p(\omega)\subseteq{\sf Per}_{\lcm(1,\dots,N)}(\omega)$.

If $\omega$ is a regular Toeplitz sequence, then $(X_\omega,S)$ is uniquely ergodic and has entropy zero
\cite[Thm.~2.5]{Downarowicz2015}. For irregular Toeplitz sequences $\omega$ a wide range of different dynamical properties of $(X_\omega,S)$ is possible, see e.g.~\cite{Williams1984}, \cite{Bulatek1992}, \cite[Ex.~5.1, 6.1]{Downarowicz2015}. For irregular $\cB$-free Toeplitz sequences $\eta$ at least one of these dynamical possibilities is excluded, namely to have a uniquely ergodic system $(X_\eta,S)$ of positive entropy (cf.~\cite{Bulatek1992}), because the Mirsky measure always exists and has entropy zero.
\end{remark}

Other examples of irregular $\cB$-free Toeplitz sequences were provided in \cite[Ex.~4.2]{KKL2016}. It is not immediately clear from that construction, however, whether those  examples may/must possess at least two invariant measures, or whether they even may/must have positive entropy.
Here we modify and tune that construction in such a way that we end up with an irregular Toeplitz sequence $\eta$ for which $(X_\eta,S)$ is uniquely ergodic. I am indebted to Stanis\l{}aw Kasjan, who provided a more systematic description of the construction from \cite[Ex.~4.2]{KKL2016}, which was instrumental in proving the following theorem.

\begin{theorem}\label{theo:2}
There are sets $\cB\subseteq\N\setminus\{1\}$ with the following properties:
\begin{compactenum}[i)]
\item The sequence $\eta=1_{\cF_\cB}$ is an irregular Toeplitz sequence.
\item $X_\eta$ is uniquely ergodic, in particular of entropy zero.
\end{compactenum}
\end{theorem}

Before we turn to the proof of Theorem~\ref{theo:1} in Section~\ref{sec:positive-entropy}, we prove Theorem~\ref{theo:2} in Section~\ref{sec:theo2}, provide a useful ``prime number characterization'' of those sets $\cB$ for which $\eta$ is a Toeplitz sequence in Section~\ref{sec:toeplitz-primes}, and prove a simplified version of Theorem~\ref{theo:1} (existence of at least two invariant measures for which $\eta$ is quasi-generic) in Proposition~\ref{prop:non-uniquely-ergodic} of Section~\ref{sec:like-Besicovitch}. The proof of one lemma, for which we rely on properties of Kolmogorov complexity, is deferred to Section~\ref{sec:complexity-entropy}.

\section{Proof of Theorem~\ref{theo:2}}\label{sec:theo2}

The starting point of the construction is a sequence $(P_k)_{k\in\N}$ of finite sets of prime numbers satisfying
\begin{enumerate}[(I)]
\item $P_1=\{2\}$,
\item\label{item:II} $\min P_{k+1}>k\,2^2\, Q_1\dots Q_k$ where $Q_j:=\prod_{q\in P_j}q$ (in particular $Q_{k+1}\geqslant 2^2\,3^{k}$) \footnote{A much weaker requirement would be possible, but this one simplifies the logic of the construction.}, and
\item\label{item:III} $d(\cM_{P_k})\geqslant 1-2^{-(k+2)}$ for $k\geqslant2$.
\end{enumerate}
It is convenient to denote the elements of $P_k$ by $q_1^{(k)},\dots,q_{t_k}^{(k)}$, so $\card(P_k)=t_k$ and $Q_k=\prod_{s=1}^{t_k}q_s^{(k)}$. Observe that $t_1=1$ and $q_1^{(1)}=2$.

Once the numbers $t_1,t_2,\dots$ are fixed we can choose the second basic ingredient of the construction, namely we fix \footnote{I am indebted to Stanis\l{}aw Kasjan, who suggested the  point of view expressed in \eqref{item:IV}. 
Note that if $R_{(i_1,\dots,i_k)}=\{r_1,r_2,\dots\}$, then the sets $R_{(i_1,\dots,i_k,s)}$ $(s=1,\dots,t_{k+1})$ may be chosen as $R_{(i_1,\dots,i_k,s)}:=\{r_{s+jt_{k+1}}:j=0,1,2,\dots\}$.
Note also that condition \eqref{item:II} has no counterpart in \cite{KKL2016}, it is introduced here to prove the unique ergodicity.}
\begin{enumerate}[(I)]\setcounter{enumi}{3}
\item \label{item:IV}
for each $k\in\N$, a partition of $\N$ into pairwise disjoint \emph{infinite} sets $R_{(i_1,\dots,i_k)}$ where $i_\ell\in\{1,\dots,t_\ell\}$ for all $\ell=1,\dots,k$, and such that
$R_{(1)}=R_{(t_1)}=\N$ and
\begin{equation*}
R_{(i_1,\dots,i_k)}=\bigcup_{s=1}^{t_{k+1}}R_{(i_1,\dots,i_k,s)}.
\end{equation*}
\end{enumerate}
After these preliminaries we define positive integers
\begin{equation*}
c_{k+j}^{(k)}=q_{i_1}^{(1)} q_{i_2}^{(2)}\cdots q_{i_k}^{(k)}\quad\text{for }j,k\in\N\text{ such that }
k+j\in R_{(i_1,\dots,i_k)},
\end{equation*}
As $\bigcup_{(i_1,\dots,i_k)}R_{(i_1,\dots,i_k)}=\Z$ in view of (\ref{item:IV}), this defines the numbers $c_{k+j}^{(k)}$ for all $j,k\in\N$. Finally let
\begin{equation*}
b_1=(q_1^{(1)})^3=2^3,\text{ and, for $k\geqslant 1$,\; } b_{k+1}=c_{k+1}^{(k)}Q_{k+1}.
\end{equation*}
It follows that
\begin{equation}\label{eq:ell_k}
\ell_k:=\lcm(b_1,\dots,b_k)=2^2\prod_{j=1}^k Q_j.
\end{equation}

\begin{lemma}
The sequence $\eta=1_{\cF_\cB}$ is an irregular Toeplitz sequence and $\cB$ is thin, i.e. $\sum_{b\in\cB}1/b<\infty$.
\end{lemma}
\begin{proof}
Let $S_k=\{b_1,\dots,b_k\}$ and define $\cA_{S_k}:=\{\gcd(\ell_k,b):b\in\cB\}$. Observe that 
\begin{equation*}
\cA_{S_k}=S_k\cup\{\gcd(\ell_k,b_{k+j}):j\geqslant1\}=S_k\cup\{q_{i_1}^{(1)}\cdots q_{i_k}^{(k)}:\exists j\geqslant1\text{ s.t. }k+j\in R_{(i_1,\dots,i_k)}\},
\end{equation*}
in particular $\limsup_{k\to\infty}(\cA_{S_k}\setminus S_k)=\emptyset$, so that $\eta$ is a Toeplitz sequence by \cite[Thm.~B]{KKL2016}. As each set $R_{(i_1,\dots,i_k)}$ is infinite, we have indeed
\begin{equation}\label{eq:ASk}
\cA_{S_k}=S_k\cup P_1\cdot P_2\cdots P_k.
\end{equation}

In order to prove that $\eta$ is irregular it suffices to show that
\begin{equation}\label{eq:irregular-condition-1}
\inf_{k\in\N}\dl\left(\cM_{\cA_{S_k}}\setminus\cM_\cB\right)>0,
\end{equation}
see \cite[Lem.~4.3]{KKL2016} together with \cite[Thm.~2.5]{Downarowicz2015}.
Observe first that $\du(\cM_\cB)\leqslant\sum_{k=1}^\infty1/b_k=2^{-3}+\sum_{k=1}^\infty1/Q_{k+1}\leqslant\frac{1}{4}$, because $Q_{k+1}\geqslant 2^2\,3^{k}$ by \eqref{item:II}. This shows in particular that $\cB$ is thin, so that the density $d(\cM_\cB)$ exists.

Next, observing \eqref{item:III} and the fact that the $P_k$ are pairwise disjoint sets of prime numbers,
\begin{equation*}
\begin{split}
d(\cM_{\cA_{S_k}})
&\geqslant 
d(\cM_{P_1\cdot P_2\cdots P_k})
=
d(\cM_{P_1}\cap\cM_{P_2}\cap\dots\cap\cM_{P_k})=d(\cM_{P_1})\cdots d(\cM_{P_k})\\
&\geqslant
\frac{1}{2}\cdot\prod_{j=2}^k(1-2^{-(j+2)})
\geqslant\frac{1}{2}\left(1-2^{-3}\right).
\end{split}
\end{equation*}
Hence the term in \eqref{eq:irregular-condition-1} is lower bounded by $\frac{1}{2}-\frac{1}{16}-\frac{1}{4}=\frac{3}{16}$.
this shows that $\eta$ is irregular.
\end{proof}

\begin{lemma}\label{lemma:lemma2}
If $k\ell_k\leqslant L<(k+1)\ell_{k+1}$, then
\begin{equation*}
\sup_{a\in\Z}\card\left(\cM_\cB\cap[a,a+L)\right)
\leqslant
L\cdot\left(d(\cM_{S_k})+\frac{4}{k}+\frac{1}{b_{k+1}}\right).
\end{equation*}
\end{lemma}

\begin{proof} Abbreviate $I:=[a,a+L)$. Then
\begin{equation*}
\card\left(\cM_\cB\cap[a,a+L)\right)
\leqslant
\card\left(\cM_{S_k}\cap I\right)+\card\left(b_{k+1}\Z\cap I\right)
+\card\left(\cM_{\cB\setminus S_{k+1}}\cap I\right),
\end{equation*}
and,
as $S_k$ has period at most $\ell_k$ and as $L\geqslant k\ell_k$,
\begin{align*}
\card\left(\cM_{S_k}\cap I\right)
&\leqslant
L\,d(\cM_{S_k})+2\ell_k
\leqslant 
L\cdot\left(d(\cM_{S_k})+2/k\right),\\
\card\left(b_{k+1}\Z\cap I\right)
&\leqslant
[L/b_{k+1}]+1
\leqslant
L\cdot\left(1/b_{k+1}+1/L\right),\qquad\text{ and}\\
\card\left(\cM_{\cB\setminus S_{k+1}}\cap I\right)
&\leqslant
1,
\end{align*}
where the last estimate is based on the following observation:
If $mb_{k+r},nb_{k+s}\in I=[a,a+L)$ for some $2\leqslant r<s$ and $m,n\in\Z$, then
$0<|mb_{k+r}-nb_{k+s}|<L$ so that
$\gcd(b_{k+r},b_{k+s})<L$. However,
\begin{equation*}
\gcd(b_{k+r},b_{k+s})
=
\gcd\left(c_{k+r}^{(k+r-1)}Q_{k+r},q_{i_1}^{(1)}q_{i_2}^{(2)}\dots q_{i_{k+s-1}}^{(k+s-1)}Q_{k+s}\right)
\geqslant
q_{i_{k+r}}^{(k+r)}
\end{equation*}
where $k+s\in R_{(i_1,\dots,i_{k+r})}$, so that, also in view of \eqref{item:II} and \eqref{eq:ell_k},
\begin{equation*}
\gcd(b_{k+r},b_{k+s})
\geqslant
q_{i_{k+r}}^{(k+r)}
\geqslant
\min P_{k+r}
>
(k+r-1)\,2^2\,Q_1\cdots Q_{k+r-1}
\geqslant
(k+1)\ell_{k+1}>L.
\end{equation*}
\end{proof}

Lemma~\ref{lemma:lemma2} implies
\begin{equation*}
\limsup_{L\to\infty}\;\sup_{a\in\Z}\frac{1}{L}\card\left(\cM_\cB\cap[a,a+L)\right)
\leqslant
d(\cM_\cB).
\end{equation*}
It follows that
\begin{equation*}
\liminf_{L\to\infty}\;\inf_{x\in X_\eta}\frac{1}{L}\sum_{k=0}^{L-1}x_k
\geqslant
1-d(\cM_\cB).
\end{equation*}
In particular, $\mu\{x\in X_\eta:x_0=1\}\geqslant 1-d(\cM_\cB)$ for each invariant measure $\mu$ on $X_\eta$. But in view of \cite[Thm.~4]{KR2015} (which owes much to Moody \cite{Moody2002}) and the correspondence between the ``sets of multiples'' and ``the cut and project'' points of view on $\cB$-free numbers (see \cite{KKL2016}, in particular Lemma~4.1), the Mirsky measure is the only invariant measure on $X_\eta$ which satisfies this inequality. Hence $(X_\eta,S)$ is uniquely ergodic.

\section{Another characterization of the case when $\eta$ is a Toeplitz sequence}
\label{sec:toeplitz-primes}

A set $\cB\subseteq\N$ is \emph{primitive}, if no number from $\cB$ divides another one. If $\cB$ is not primitive, there is always a unique primitive subset $\cB'\subseteq\cB$ such that $\cM_\cB=\cM_{\cB'}$.

For $k\in\N$ let $\cB/k:=\{\frac{b}{k}: b\in\cB, k\mid b\}$. Observe that $\cB/k=\{1\}$ if and only if $k\in\cB$, whenever $\cB$ is primitive.
\begin{lemma}\label{lemma:lemma1}
Suppose that $\cB$ is primitive and
let $k\in\N\setminus\cB$. Then $\cB/k$ contains no infinite pairwise coprime subset if and only if there is a finite set of primes $P_k$ such that $\cB/k\subseteq \cM_{P_k}$.
\end{lemma}
\begin{proof}
$\cB/k$ contains an infinite pairwise coprime subset if and only if $\cB/k\not\subseteq\cM_C$ for all finite sets $C\subseteq\N\setminus\{1\}$ \cite[Thm.~3.7]{BKKL2015} 
\footnote{As stated in \cite{KKL2016}, Theorem 3.7 of \cite{BKKL2015} has a minor flaw (which is completely irrelevant for the results in that paper): The implication ``(f)$\Rightarrow$(g)'' does not hold, when $1\in\cB$. 
As we apply this implication to the set $\cB/k$, we make sure that $1\not\in\cB/k$ 
by requiring 
$k\not\in\cB$.}, 
and the latter is equivalent to $\cB/k\not\subseteq\cM_P$ for all finite sets $P$ of primes. The claim of the lemma is just the negation of this equivalence.
\end{proof}

\begin{proposition}\label{prop:top-reg}
Suppose that $\cB$ is primitive.
The sequence $\eta=1_{\cF_\cB}$ is a Toeplitz sequence if and only if for every $k\in\N\setminus\cB$ there is a finite set $P_k$ of primes such that $\cB/k\subseteq\cM_{P_k}$.
\end{proposition}
\begin{proof}
$\eta$ is a Toeplitz sequence if and only if there are no $k\in\N$ and no infinite pairwise coprime set $\cA\subseteq\N\setminus\{1\}$ such that $k\cA\subseteq\cB$
\cite[Thm.~B]{KKL2016}.
As $\cB$ is primitive, there can never be $k\in\cB$ and an infinite pairwise coprime set $\cA\subseteq\N\setminus\{1\}$ such that $k\cA\subseteq\cB$. Hence
$\eta$ is a Toeplitz sequence if and only if there are no $k\in\N\setminus\cB$ and no infinite pairwise coprime set $\cA\subseteq\N\setminus\{1\}$ such that $k\cA\subseteq\cB$.
 But $k\cA\subseteq\cB$ is equivalent to $\cA\subseteq\cB/k$, so that an application of Lemma~\ref{lemma:lemma1} finishes the proof.
\end{proof}

\section{Non-uniquely ergodic $\cB$-free Toeplitz sequences}\label{sec:like-Besicovitch}

For each $\varepsilon>0$, Besicovitch \cite{Besicovitch1935} provided an example of a primitive set $G\subseteq\N\setminus\{1\}$ such that the lower asymptotic density
$\dl(\cM_G)<\varepsilon$, while the upper asymptotic density of this set is $\du(\cM_G)>\frac{1}{2}$.
\footnote{Besicovitch denotes the set $\cM_G$ by $H$.} 
\begin{remark}\label{remark:Besicovitch-proximal} 
The set $\cM_G$ in Besicovitch's example contains arbitrarily long intervals $[T,2T)$.
Since, by the Bertrand postulate (proved by Tchebichef \cite[pp.~371--382]{Tchebichef1852}) each such interval contains at least one prime number, the set $G$  contains infinitely many prime numbers, and so the corresponding subshift $X_\eta=X_{1_{\cF_G}}$ is proximal \cite[Thm.~B]{BKKL2015}.
Its ``tautification'' $G'$ is proximal, as well \cite[Thm.~4.5]{BKKL2015} (see also \cite[Sec.~2.3]{KKL2016}), so the corresponding $X_{\eta'}$ is hereditary \cite[Thm.~3]{Keller2017}.
Since $\cF_{G'}$ has the same positive logarithmic density as $\cF_G$, $X_{\eta'}$ and hence also $X_\eta$ have positive entropy (equal to this logarithmic density)
\cite[Prop.~K and Cor.~1.7]{BKKL2015}.
\end{remark}

Our goal is to modify Besicovitch's construction in several respects by defining
primitive sets $\cB\subseteq\N\setminus\{1\}$ such that 
\begin{compactitem}[-]
\item $\cM_\cB\subseteq\cM_G$, so $\dl(\cM_\cB)\leqslant\dl(\cM_G)<\varepsilon$,
\item $\eta=1_{\cF_\cB}\in\{0,1\}^\Z$
is a Toeplitz sequence,
\item there is at least one invariant measure of positive entropy for which $\eta$ is not quasi-generic, and
\item depending on details of the construction, $\eta$ is generic for the Mirsky measure, or it is quasi-generic for some measure of positive entropy (and, of course, for the Mirsky measure).
\end{compactitem}

We start by recalling the essentials of Besicovitch's construction, following more or less the outline in \cite[second part of Thm.~0.1]{Hall1996}:
Take positive numbers $\varepsilon,\varepsilon_i$ $(i=1,2,\dots)$ such that
\begin{equation*}
\varepsilon<\frac{1}{4}, \quad\sum_{i=1}^\infty\varepsilon_i<\frac{\varepsilon}{2}.
\end{equation*}
Denote $E_T:=\cM_{[T,2T)}$ and write $e(T)$ for the asymptotic density of $E_T$.
As $E_T$ is periodic, there are numbers $\lambda(T)$ such that 
the mean density of the set $E_T$ on any interval of more than $\lambda(T)$ consecutive integers is $<2e(T)$.

Define integers $1=T_0<T_1<T_2<T_3<\dots$ so that
\begin{align}
&e(T_1)<\varepsilon_{1}\nonumber\\
T_2>\lambda(T_1),\quad& e(T_2)<\varepsilon_{2}\label{eq:Bes-constr}\\
T_3>\lambda(T_2),\quad& e(T_3)<\varepsilon_{3}\nonumber\\
&\vdots\nonumber
\end{align}
These inductive choices are possible, because of 
Erd\H{o}s' result \cite{Erdos1935a} that $\lim_{T\to\infty}e(T)=0$. \footnote{
Besicovitch \cite{Besicovitch1935} used his weaker Theorem~1, which asserts that
$e(2^1)+e(2^2)+\dots+e(2^n)=o(n)$.}
Observe that, given $T_1,\dots,T_k$, the index $T_{k+1}$ can be chosen arbitrarily large. We will make use of this freedom of choice in the sequel.
\footnote{Instead of the constraint $T_k>\lambda(T_{k-1})$ in \eqref{eq:Bes-constr}, Besicovitch requires more explicitly $2^{i_k}>(2^{i_{k-1}+1})!$.}

Besicovitch's set $G$ is then defined as
\begin{equation}\label{eq:Besicovitch-1}
G=\bigcup_{k=1}^\infty \left[T_k,2T_{k}\right)\setminus(E_{T_1}\cup\dots\cup E_{T_{k-1}}),
\end{equation}
and obviously $[T_k,2T_{k})\subseteq\bigcup_{j=1}^\infty E_{T_j}=\cM_G$ for all $k$. As Besicovitch observed,
\begin{equation*}
\du(\cM_G)
\geqslant
\limsup_{k\to\infty}(2T_k)^{-1}\card(\cM_G\cap[1,2T_k))
\geqslant
\limsup_{k\to\infty}(2T_k)^{-1}\card[T_k,2T_k)
=
\frac{1}{2},
\end{equation*}
whereas
\begin{equation}\label{eq:Besicovitch-3}
\begin{split}
\dl(\cM_G)
&\leqslant
\liminf_{k\to\infty}T_k^{-1}\card(\cM_G\cap[1,T_k))
=
\liminf_{k\to\infty}T_k^{-1}\card((E_{T_1}\cup\dots\cup E_{T_{k-1}})\cap[1,T_k))\\
&\leqslant
\liminf_{k\to\infty}\sum_{j=1}^{k-1}\card(E_{T_j}\cap[1,T_k))
\leqslant
\sum_{j=1}^{k-1}2e(T_j)
<
\sum_{j=1}^\infty2\varepsilon_j
<\varepsilon.
\end{split}
\end{equation}

We now proceed to introduce additional constraints to the choice of the indices $T_k$ and to construct a set $\cB\subseteq\N\setminus\{1\}$ with the following properties:
\begin{enumerate}[(I)]
\item For every $j\in\N\setminus\cB$ there is a finite set $P_j$ of primes such that $\cB/j\subseteq\cM_{P_j}$ (see also Section~\ref{sec:toeplitz-primes}),
\item $\dl(\cM_\cB)<\varepsilon$, and
\item $\du(\cM_\cB)\geqslant\frac{1}{2}-2\varepsilon$.
\end{enumerate}
To this end assume
that integers $1=T_0<T_1<\dots<T_{k}$, positive integers $L_1,\dots,L_k$,
and finite sets $P_1,\dots,P_{T_k-1}$ of prime numbers
are chosen such that (setting $T_{-1}=1$)
\begin{compactenum}[(A)]
\item the following strengthening of Besicovitch's constraints \eqref{eq:Bes-constr} is satisfied for $i=1,\dots,k$:
\begin{equation*}
T_i\geqslant L_i>\lambda(T_{i-1}),\quad e(T_i)<\varepsilon_i,
\end{equation*}
\item
$\card(j\cdot\cF_{P_j}\cap[T,2T))\leqslant2d(j\cdot\cF_{P_j})\cdot T$ for all 
$j\in[1,T_{k-1})$ and $T\geqslant T_k$,
\item $\Spec([1,2T_k))\subseteq P_j$  for all $j\in[T_{k-1},T_k)$, and
\item $d(j\cdot\cF_{P_j})<\varepsilon\cdot2^{-(j+1)}$ for all $j\in[1,T_k)$.
\end{compactenum}

Observe first that conditions (A) -- (D) are empty and hence trivially satisfied for $k=0$.

Now we choose $T_{k+1}\geqslant L_{k+1}>\max\{T_k,\lambda(T_k)\}$ and sets $P_j$ $(T_k\leqslant j<T_{k+1})$ inductively in such a way that (A) -- (D) hold for $k+1$ instead of $k$:
First we make sure that $T_{k+1}$ is large enough to satisfy (A) and (B) for $k+1$.
(For property (B) note that the sets $j\cdot\cF_{P_j}$ are periodic.)
Then we choose the additional $P_j$ big enough such that also (C) and (D) are satisfied for $k+1$.

For the next step of the construction we fix, for all $k\in\N$,
sets $J_k\subseteq[T_k,T_k+L_k)$ (with additional properties to be specified below), and define
\begin{align*}
\cF_{P_j}^*:=&\cF_{P_j}\setminus\{1\}\\
F:=&\bigcup_{j=1}^\infty j\cdot\cF_{P_j}^*\\
E_j':=&\cM_{J_j\setminus F}\quad(j\in\N)\\
\cB_n:=&\bigcup_{k=1}^n\left(J_k\setminus F\right)\setminus\bigcup_{j=1}^{k-1}E_j'\quad(n\in\N)\\
\cB:=&\bigcup_{n=1}^\infty\cB_n.
\end{align*}

\begin{lemma}\label{lemma:cB-properties}
\begin{compactenum}[a)]
\item $\cB$ is primitive by construction.
\item $\cB\cap F=\emptyset$ by construction.
\item$\cB/j\subseteq\cM_{P_j}$ for every $j\in\N\setminus\cB$.
\item $\eta=1_{\cF_\cB}$ is a Toeplitz sequence.
\end{compactenum}
\end{lemma}
\begin{proof}
a) and b):\;Obvious.\\
c)\;Let $b\in\cB/j$. Then $jb\in\cB$, whence $jb\not\in F$ by assertion b). In particular, $jb\not\in j\cdot\cF_{P_j}^*$, i.e.~$b\not\in\cF_{P_j}^*$. Hence $b=1$ or $b\in\cM_{P_j}$. But $b\neq1$ since $j\not\in\cB$.\\
d)\;$\eta$ is a Toeplitz sequence by Proposition~\ref{prop:top-reg} and assertions a) and c).
\end{proof}

\begin{lemma}\label{lemma:cB-lemma}
\begin{compactenum}[a)]
\item $\cM_{\cB_n}=\bigcup_{j=1}^n E_j'$.
\item $\cM_\cB\subseteq\cM_G$.
\item $\cM_\cB\cap J_k\supseteq J_k\setminus F$ for all $k$.
\item $\dl(\cM_\cB)<\varepsilon$.
\end{compactenum}
\end{lemma}
\begin{proof}
a)\;
For $n=1$ we have $\cB_1=J_1\setminus F$, whence $\cM_{\cB_1}=E_1'$. It follows inductively that
\begin{equation*}
\begin{split}
\cM_{\cB_{n+1}}
=&
\cM_{\cB_n}\cup\cM_{\left(J_{n+1}\setminus F\right)\setminus \cM_{\cB_n}}
=
\cM_{\cM_{\cB_n}\cup \left(\left(J_{n+1}\setminus F\right)\setminus \cM_{\cB_n}\right)}\\
=&
\cM_{\cM_{\cB_n}\cup \left(J_{n+1}\setminus F\right)}
=
\cM_{\cB_n}\cup E_{n+1}'
=\bigcup_{j=1}^{n+1}E_j'.
\end{split}
\end{equation*}
b)\;
In view of assertion a), $\cM_\cB=\bigcup_{j=1}^\infty E_j'\subseteq\bigcup_{j=1}^\infty E_{T_j}=\cM_G$, see \eqref{eq:Besicovitch-1}.\\
c)\;
$\cM_\cB\cap J_k=\bigcup_{j=1}^\infty E_j'\cap J_k\supseteq\bigcup_{j=1}^\infty\left(J_j\setminus F\right)\cap J_k=J_k\setminus F$.\\
d)\; follows from b) and~\eqref{eq:Besicovitch-3}.
\end{proof}

We will use the following two estimates:
\begin{lemma}\label{lemma:card-estimates} 
For all $k\in\N$,
\begin{compactenum}[a)]
\item $\card(F\cap[T_k,T_k+L_k))\leqslant\varepsilon L_k$, 
\item $\card\left(\bigcup_{j=1}^{k-1}E_j'\setminus(J_k\setminus F)\cap[T_k,T_k+L_k)\right)
\leqslant
2\varepsilon L_k$.
\end{compactenum}
\end{lemma}
\begin{proof}
a)\;
\begin{equation*}
\begin{split}
&\card(F\cap[T_k,T_k+L_k))\\
&\leqslant
\sum_{j=1}^\infty\card([T_k,T_k+L_k)\cap j\cdot\cF_{P_j}^*)\\
&=
\sum_{j=1}^{T_{k-1}-1}\card([T_k,T_k+L_k)\cap j\cdot\cF_{P_j}^*)
+
\sum_{j=T_{k-1}}^{T_k-1}\card([T_k,T_k+L_k)\cap j\cdot\cF_{P_j}^*)\\
&\hspace*{3cm}+
\sum_{j=T_k}^\infty\card([T_k,T_k+L_k)\cap j\cdot\cF_{P_j}^*)\\
&\leqslant
\sum_{j=1}^{T_{k-1}-1}2d(j\cdot\cF_{P_j}^*)L_k
+\sum_{j=T_{k-1}}^{T_k-1}0
+\sum_{j=T_k}^{\infty}0
<
\sum_{j=1}^{T_{k-1}-1}\varepsilon2^{-j}L_k<\varepsilon L_k.
\end{split}
\end{equation*}
Here the first ``$0$-sum'' is due to property (C), and for the second ``$0$-sum'' one
only needs to observe that $1\not\in\cF_{P_j}^*$.
The final estimate uses property (D).\\
b)\;
\begin{equation*}
\card\left(\bigcup_{j=1}^{k-1}E_j'\setminus(J_k\setminus F)\cap[T_k,T_k+L_k)\right)
\leqslant
\sum_{j=1}^{k-1}\card\left(E_j'\cap[T_k,T_k+L_k)\right)
\leqslant
\sum_{j=1}^{k-1}2e(T_j)L_k
\leqslant
2\varepsilon L_k.
\end{equation*}
\end{proof}

\begin{proposition}\label{prop:non-uniquely-ergodic}
There are primitive sets $\cB$ with the following properties:
\begin{compactenum}[i)]
\item The sequence $\eta=1_{\cF_\cB}$ is a Toeplitz sequence.
\item $\cB$ is not a Besicovitch set.
\item The sequence $\eta$ is  quasi-generic for at least two measures.
\end{compactenum}
\end{proposition}
\begin{proof}
Let $J_k=[T_k,2T_k)$ for all $k$. Lemma~\ref{lemma:cB-lemma}d) shows that $\eta$ is a Toeplitz sequence, and Lemma~\ref{lemma:cB-lemma}c) and Lemma~\ref{lemma:card-estimates}a) imply
\begin{equation*}
\card(\cM_\cB\cap[1,2T_k))
\geqslant
\card([T_k,2T_k)\setminus F)
\geqslant
T_k-\card(F\cap[T_k,2T_k))
\geqslant
(1-\varepsilon)T_k
\end{equation*}
for every $k$, so that in particular $\du(\cM_\cB)>\frac{1}{2}-\varepsilon$. 
Combined with Lemma~\ref{lemma:cB-lemma}d)
this shows that $\cB$ is not a Besicovitch set and that $\eta$ is not generic for any measure, so it is quasi-generic for at least two measures.
\end{proof}

\section{Positive entropy}\label{sec:positive-entropy}

For the proof of Proposition~\ref{prop:non-uniquely-ergodic} we made the straightforward choice $J_k=[T_k,2T_k)$. In order to control the entropy of the measures we construct, we will have to make more subtle choices for the sets $J_k\subseteq[T_k,2T_k)$, and in order to include also measures, for which $\eta$ is not quasi-generic, we replace the intervals $[T_k,2T_k)$ by more flexible intervals $[T_k,T_k+L_k)$. The choice of the sets $J_k$ is based on the following lemma, which might be folklore among specialists, but which I could not locate in the literature. So I provide a proof based on properties of Kolmogorov complexity in Section~\ref{sec:complexity-entropy}.

\begin{lemma}\label{lemma:entropy-complexity}
Let $\varepsilon\in(0,\frac{1}{2})$. There is a constant $L_\varepsilon>0$ such that for all
$L\geqslant L_\varepsilon$ and $\gamma\in(0,1/2-\varepsilon)$ there is a word $w_{L,\gamma}\in\{0,1\}^L$ with the following properties:
For each $n>0$ and each $\kappa>0$ there is $\ell_{n,\kappa}>0$  such that, for
all sets $A,B\subseteq\{1,\dots,L\}$ with $d_A,d_B<\varepsilon$, $\Phi(d_A),\Phi(d_B)<\frac{1}{4}\kappa$ and $w_{L,\gamma}\cdot1_{A^c}\cdot 1_B=0$, 
\begin{eqnarray}
(\gamma-\varepsilon) L\leqslant\sum_{i=1}^L(w_{L,\gamma}\cdot 1_{A^c}+1_B)_i\leqslant(\gamma+2\epsilon)L\quad\text{and}
\label{eq:density-bound}\\
\frac{1}{n}H_n(w_{L,\gamma}\cdot 1_{A^c}+1_B)
\geqslant \Phi(\gamma)-\kappa\quad\text{if }L\geqslant\ell_{n,\kappa}.
\end{eqnarray}
\end{lemma}

We now describe how to choose the $J_k$ in order to get a measure of positive entropy for which $\eta$ is quasi-generic. 
So let $\varepsilon\in(0,\frac{1}{4})$ and choose $\gamma\in(\varepsilon,\frac{1}{2}-\varepsilon)$.
Fix also some number $\kappa\in(0,\varepsilon)$. For each $n\in\N$ and all indices $k$ such that $L_k\geqslant\ell_{n,\kappa}$ we choose a word $w_k=w_{L_k,\gamma}\in\{0,1\}^n$ as in Lemma~\ref{lemma:entropy-complexity}.

For any $w\in\{0,1\}^L$ denote $J(w):=\{i\in[1,L]:w_i=1\}$.
Define the sets $J_k$ for our construction,
\begin{equation}
J_k:=J(w_k)+T_k-1\subseteq[T_k,T_k+L_k).
\end{equation}
Fix a subsequence $(T_{k_i})_i$ for which
\begin{compactitem}
\item the sequence $\left(\frac{1}{T_{k_i}}\sum_{j=0}^{T_{k_i}-1}\delta_{S^j\eta}\right)_i$
converges weakly to some invariant measure 
$\nu_1$, and
\item the sequence $\left(\frac{1}{L_{k_i}}\sum_{j=T_{k_i}}^{T_{k_i}+L_{k_i}-1}\delta_{S^j\eta}\right)_i$
converges weakly to some invariant measure $\nu_2$.
\end{compactitem}
Then,
\begin{compactitem}
\item if $L_k=T_k$, the sequence $\left(\frac{1}{T_{k_i}+L_{k_i}}\sum_{j=0}^{T_{k_i}+L_{k_i}-1}\delta_{S^j\eta}\right)_i$
converges weakly to the invariant measure $\nu=\frac{1}{2}(\nu_1+\nu_2)$, and
\item  if $L_k/T_k\to0$, the sequence $\left(\frac{1}{T_{k_i}+L_{k_i}}\sum_{j=0}^{T_{k_i}+L_{k_i}-1}\delta_{S^j\eta}\right)_i$
converges weakly to the invariant measure $\nu_1$.
\end{compactitem}
Without loss of generality we may assume that $(T_{k_i})_i$ is the full sequence $(T_k)_k$ -- this just eases the notation. 

%

\begin{lemma}\label{lemma:pos-entropy}
We have the following lower bound for the Kolmogorov-Sinai entropy of $(X_\eta,S,\nu_2)$:
$$h_{\nu_2}(S)\geqslant\Phi(\gamma)-4\Phi(2\varepsilon).$$
\end{lemma}
\begin{proof}
For each $k\in\N$,
\begin{equation}\label{eq:MB-intersection}
\begin{split}
\cM_\cB\cap[T_k,T_k+L_k)
&=
\cM_{\cB_k}\cap[T_k,T_k+L_k)\\
&=
\bigcup_{j=1}^{k-1}E_j'\cap[T_k,T_k+L_k)\cup \left((J_k\setminus F)\cap[T_k,T_k+L_k)\right)\\
&=
J_k\setminus\left(F\cap [T_k,T_k+L_k)\right)\cup
\bigcup_{j=1}^{k-1}\left(E_j'\setminus(J_k\setminus F)\right)\cap[T_k,T_k+L_k)\\
&=:
(J(w_k)+T_k-1)\setminus (A_k+T_k-1)\cup (B_k+T_k-1)\\
&=
J(w_k\cdot 1_{A_k^c}+1_{B_k})
\end{split}
\end{equation}
with sets $A_k,B_k\subseteq[1,L_k]$ such that $B_k\cap (J(w_k)\setminus A_k)=\emptyset$, to which we want to apply Lemma~\ref{lemma:entropy-complexity} -- with $2\varepsilon$ instead of $\varepsilon$ and $\kappa=4\Phi(2\epsilon)$. To this end observe that Lemma~\ref{lemma:card-estimates} implies
\begin{eqnarray*}
d_{A_k}
=
L_k^{-1}\card(F\cap[T_k,T_k+L_k))
\leqslant
\varepsilon\quad\text{and}\hspace*{1cm}\\
d_{B_k}
=
L_k^{-1}\card\left(\bigcup_{j=1}^{k-1}E_j'\setminus(J_k\setminus F)\cap[T_k,T_k+L_k)\right)
\leqslant
2\varepsilon,
\end{eqnarray*}
in particular also $\Phi(d_{A_k}),\Phi(d_{B_k})<\Phi(2\varepsilon)=\frac{\kappa}{4}$.
Hence, Lemma~\ref{lemma:entropy-complexity} shows that
for each $n\in\N$ there is $k_n>0$  such that, for all $k\geqslant k_n$,
\begin{eqnarray}
(\gamma-2\varepsilon) L_k\leqslant\sum_{i=1}^{L_k}(w_k\cdot 1_{A_k^c}+1_{B_k})_i\leqslant(\gamma+4\epsilon)L_k\quad\text{and}
\label{eq:density-bound*}\\
\frac{1}{n}H_n(w_k\cdot 1_{A_k^c}+1_{B_k})
\geqslant \Phi(\gamma)-\kappa.\hspace*{1cm}\label{eq:entropy_lower_bound}
\end{eqnarray}

So fix $n\in\N$. For each cylinder set $[u]$ determined by $u\in\{0,1\}^n$ we have
\begin{equation*}
\nu_2([u])
=
\lim_{k\to\infty}\frac{1}{L_k}\sum_{\ell=T_k}^{T_k+L_k-n}1_{[u]}(S^\ell\eta)
=
\lim_{k\to\infty}\frac{1}{L_k}\card\{\ell\in[T_k,T_k+L_k-n]:\eta_{[\ell,\ell+n-1]}=u\}.
\end{equation*}
It follows from~\eqref{eq:MB-intersection} and~\eqref{eq:entropy_lower_bound}
that $H_n(\nu_2)$, the entropy of $\nu_2$ on blocks of length $n$, can be estimated by
\begin{equation*}
\frac{1}{n}H_n(\nu_2)
\geqslant
\Phi(\gamma)-\kappa
=
\Phi(\gamma)-4\Phi(2\epsilon),
\end{equation*}
so that
$h_{\nu_2}(S)\geqslant \Phi(\gamma)-4\Phi(2\varepsilon)$.
\end{proof}
\begin{proof}[Proof of Theorem~\ref{theo:1}]
(ii)\;By Lemma~\ref{lemma:pos-entropy}, $h_{\nu_2}(S)\geqslant\Phi(\gamma)-4\Phi(2\epsilon)$ is strictly positive if $\gamma>\Phi^{-1}(4\Phi(2\epsilon))$, which can easily be achieved for small enough $\varepsilon>0$. \\
(i)\;
$\eta$ is a Toeplitz sequence by Lemma~\ref{lemma:cB-properties}d). It is irregular, because $X_\eta$ is not uniquely ergodic by assertion (ii).\\
(iii)\;
(a)\;Choose $T_k=L_k$. Then $\eta$ is quasi-generic for the invariant measure $\nu=\frac{1}{2}(\nu_1+\nu_2)$, and $h_\nu(S)\geqslant\frac{1}{2}h_{\nu_2}(S)>0$ as above. 
\\
(b)\;Choose $T_k=k^2L_k$. Then $\nu_2$ is a measure of positive entropy as before, but the set $\cB$ is Besicovitch: 
\begin{equation*}
\sum_{b\in\cB}\frac{1}{b}
\leqslant
\sum_{k=1}^\infty\sum_{j=T_k}^{T_k+L_k-1}\frac{1}{j}
\leqslant
1+\sum_{k=1}^\infty\sum_{j=T_k+1}^{T_k+L_k}\frac{1}{j}
\leqslant
1+\sum_{k=1}^\infty\log\frac{T_k+L_k}{T_k}
\leqslant
1+\sum_{k=1}^\infty\log\left(1+\frac{1}{k^2}\right)<\infty.
\end{equation*}
\end{proof}

\section{On Kolmogorov complexity and entropy}\label{sec:complexity-entropy}

Only Lemma~\ref{lemma:entropy-complexity} from this section will be used in the sequel. It is formulated just in terms of entropy, and the reader who considers it as folklore should skip this section.

Very loosely speaking, the Kolmogorov complexity $C(w)$ of a word $w\in\{0,1\}^*$ is the length of the shortest binary code that can serve as a program for a universal Turing machine to print the word $w$ on its output tape and then to stop. Of course this definition depends on the choice of the particular Turing machine, but it can be shown that for any two different universal Turing machines there exists a constant such that the difference of complexities defined with respect to these two machines does not exceed this constant for any word $w$ of any length. The monograph \cite{Li-Vitanyi} provides a precise and detailed introduction to Kolmogorov complexity and other variants of algorithmic complexity and their relation to entropy and coding, and we will refer to notation and results from this book throughout this section.

A general pitfall when dealing with algorithmic complexity is that (in)equalities which one might expect when one does not think too much about the details of their proofs, hold only up to a constant or even logarithmic (logarithm of the word length) error term. One of the reasons is that the transitions between consecutive words on the same input tape of the Turing machine must be recognizable, another one that sometimes the word length must be provided as additional information to the Turing machine to make the intended algorithm work. This can be dealt with properly by introducing variants of Kolmogorov complexity like the prefix complexity $K(w)$ in \cite[Sec.~3.1]{Li-Vitanyi}.
It should not come as a surprise that $C(w)$ and $K(w)$ differ only by a logarithmic
(in the word length) term.
As logarithmic terms do not influence our arguments, we will provide only ``naive'' proofs, whenever complexity is involved.

Denote by $\Phi:[0,1]\to[0,\log 2]$, $\phi(t)=-t\log_2(t)-(1-t)\log_2(1-t)$ the binary entropy function.
The following lemma is folklore:

\begin{lemma}\label{lemma:high-complexity-words}
Let $\varepsilon>0$. There is a constant $L_\varepsilon$
such that for each $L\geqslant L_\varepsilon$
and each $\gamma\in(0,1/2-\varepsilon)$
there is some $w_{L,\gamma}\in\{0,1\}^L$ such that
\begin{equation*}
\gamma L\leqslant\sum_{i=1}^L(w_{L,\gamma})_i\leqslant(\gamma+\varepsilon)L
\quad\text{and}\quad
C(w_{L,\gamma})\geqslant\Phi(\gamma) L.
\end{equation*}
\end{lemma}
\begin{proof}
For large enough $L$ (``large'' depending only on $\varepsilon$) we can fix $k\in\N$ such that $\gamma+\varepsilon/2<k/L<\gamma+\varepsilon$. Hence there are at least $L\choose k$ words $w\in\{0,1\}^L$ with
$(\gamma+\varepsilon/2) L\leqslant\sum_{i=1}^Lw_i\leqslant(\gamma+\varepsilon)L$.
At least one of these words has complexity $C(w)\geqslant\log_2{L\choose k}-1$ \cite[Thm.~2.2.1]{Li-Vitanyi}, and one can estimate that this is bounded from below by 
$L\Phi(\gamma)$ when $L\geqslant L_\varepsilon$ for some suitable $L_\varepsilon$.
\footnote{${L\choose k}\geqslant 2^{L(\Phi(k/L)+o(1))}$. Hence 
$\log_2{L\choose k}-1\geqslant L(\Phi(\gamma+\varepsilon/2)+o(1))\geqslant L\Phi(\gamma)$ if $L$ is larger than a constant depending only on $\varepsilon$.}
\end{proof}

We fix some notation.
\begin{itemize}[-]
\item For $0<n<L$ let $m:=[(L-(n-1))/n]$ so that $L':=(m+1)n-1\leqslant L<(m+2)n-1$.
\item Let $0<n<L$ and $w\in\{0,1\}^L$.
\begin{compactenum}[(1)]
\item
For $s\in\{0,\dots,n-1\}$ denote by $H_n^s(w)$ the entropy of the empirical distribution of blocks of length $n$ in the sample $(w_{[jn+s+1,jn+s+n]})_{j=0,\dots,m-1}$. (These are the non-overlapping sub-words of $w$ with length $n$ starting at position $s$, except possibly for the last one.)
\item Denote by $H_n(w)$ the entropy of the empirical distribution of blocks of length $n$ in the sample $(w_{[j+1,j+n]})_{j=0,\dots,L-n}$. (These are all sub-words of $w$ with length $n$.)
\end{compactenum}
\item For $A\subseteq \{1,\dots,L\}$ denote $d_A:=\card(A)/L$.
\end{itemize}
\begin{lemma}\label{lemma:entropy-inequality}
$H_n(w)\geqslant \frac{1}{n}\sum_{s=0}^{n-1}H_n^s(w)-q_n(L)$, where $q_n(L)\to0$ as $L\to\infty$.
\end{lemma}
\begin{proof}
Denote by $w'$ the restriction of the word $w$ to the indices $[1,L']$. Then the collection of length-$n$ subword of $w'$ is the disjoint union of the samples from item (1), so that $H_n(w')\geqslant \frac{1}{n}\sum_{s=0}^{n-1}H_n^s(w)$, because the entropy function (as a function on probability vectors) is concave. So it remains to estimate the difference $H_n(w')-H_n(w)$. As $L-L'<n$, a crude estimate can use the fact that the relative frequencies of any block $u\in\{0,1\}^n$ in $w$ and $w'$ can differ by at most $(n-1)/L'<n/((m+1)n)=1/(m+1)$. Hence, the contribution of each single block to the entropy can change by at most $\varphi(1/(m+1))$, where we use that the function $\varphi(x)=-x\log_2x$ is concave and increasing on the interval $[0,e^{-1}]$. It follows that $|H_n(w')-H_n(w)|\leqslant 2^n\varphi(1/(m+1))\leqslant2^n\varphi(\frac{n}{L-(n-1)})=:q_n(L)$ and $q_n(L)\to0$ as $L\to\infty$.
\end{proof}

\begin{proof}[Sketch of a proof of Lemma~\ref{lemma:entropy-complexity} using Kolmogorov complexity]
The inequalities in \eqref{eq:density-bound} are obvious. We turn to the lower bound for the entropy.
Let $w=w_{L,\gamma}$.
It is intuitively clear that 
\begin{equation*}
C(w)\leqslant C(w\cdot 1_{A^c}+1_B)+C(1_B)+C(w\cdot 1_A)+O(\log L).
\end{equation*}
Just observe that $(w\cdot 1_{A^c}+1_B)-1_B+w\cdot 1_A=w$.
It follows from \cite[Thm.~2.8.1]{Li-Vitanyi} that, for each $n>0$, there is a sequence $\epsilon_n(1)>\epsilon_n(2)>\dots\searrow 0$
 such that, for all $s\in\{0,\dots,n-1\}$,
\begin{equation*}
\begin{split}
C(w\cdot 1_{A^c}+1_B)
&\leqslant
m(H_n^s(w\cdot 1_{A^c}+1_B)+\epsilon_n(m))+2n
\\
C(1_B)
&\leqslant
L(H_1(1_B))+\epsilon_1(L))\\
C(w\cdot 1_{A})
&\leqslant
L(H_1(w\cdot 1_A)+\epsilon_1(L)).
\end{split}
\end{equation*}
But $H_1(1_B)\leqslant\Phi(d_B)$, $H_1(w\cdot 1_A)\leqslant\Phi(d_A)$, 
and $\frac{1}{n}\sum_{s=0}^{n-1}H_n^s(w\cdot 1_{A^c}+1_B)\leqslant H_n(w\cdot 1_{A^c}+1_B)+q_n(L)$ by Lemma~\ref{lemma:entropy-inequality},
so that
\begin{equation*}
C(w)\leqslant
L\left(\frac{1}{n}H_n(w\cdot1_{A^c}+1_B)+\frac{m}L\epsilon_n(m)+q_n(L)+\Phi(d_B)+\Phi(d_A)+\epsilon_1(L)\right).
\end{equation*}
By Lemma~\ref{lemma:high-complexity-words},
\begin{equation*}
C(w)\geqslant L\Phi(\gamma)-1.
\end{equation*}
Hence
\begin{equation*}
\begin{split}
\frac{1}{n}H_n(w\cdot1_{A^c}+1_B)
&\geqslant
\Phi(\gamma)-\frac{1}{L}-q_n(L)-\frac{\epsilon_n(L/(n+1))}{n}-\Phi(d_A)-\Phi(d_B)-\epsilon_1(L)\\
&=
\Phi(\gamma)-\kappa/2-\rho_n(L),
\end{split}
\end{equation*}
where $\rho_n(L):=\frac{1}{L}+q_n(L)+\frac{\epsilon_n(L/(n+1))}{n}+\epsilon_1(L)\searrow 0$ as $L\to\infty$ for each fixed $n$. To finish the proof,
choose $\ell_{n,\kappa}\geqslant L_\varepsilon$ so large that $\rho_n(L)\leqslant\frac{\kappa}{2}$ for $L\geqslant \ell_{n,\kappa}$.
\end{proof}

%

\begin{thebibliography}{10}

\bibitem{Besicovitch1935}
A.~S. Besicovitch.
\newblock {On the density of certain sequences of integers}.
\newblock {\em Mathematische Annalen}, 110(1):336--341, 1935.

\bibitem{Bulatek1992}
W.~Bu{\l}atek and J.~Kwiatkowski.
\newblock {Strictly ergodic Toeplitz flows with positive entropies and trivial
  centralizers}.
\newblock {\em Studia Mathematica}, 103(2):133--142, 1992.

\bibitem{DavenportErdos1937}
H.~Davenport and P.~Erd\H{o}s.
\newblock {On sequences of positive integers}.
\newblock {\em Acta Arithmetica}, 2:147--151, 1937.

\bibitem{Downarowicz2015}
T.~Downarowicz and S.~Kasjan.
\newblock {Odometers and Toeplitz systems revisited in the context of Sarnak's
  conjecture}.
\newblock {\em Studia Mathematica}, 229(1):45--72, 2015.

\bibitem{Dymek2017}
A.~Dymek.
\newblock {Automorphisms of Toeplitz $\cB$-free systems}.
\newblock {\em Bull. Polish Acad. Sci. Math.}, 65:139--152, 2017.

\bibitem{BKKL2015}
A.~Dymek, S.~Kasjan, J.~Ku\l{}aga-Przymus, and M.~Lema\'n{}czyk.
\newblock {${\cB}$-free sets and dynamics}.
\newblock {\em Trans. Amer. Math. Soc.}, 370(8):5425--5489, 2018.

\bibitem{Erdos1935a}
P.~Erd\H{o}s.
\newblock {Note on sequences of integers no one of which is divisible by any
  other}.
\newblock {\em J.~London Math.~Soc.}, 10:126--128, 1935.

\bibitem{Hall1996}
R.~R. Hall.
\newblock {\em {Sets of Multiples}}, volume 118 of {\em Cambridge Tracts in
  Mathematics}.
\newblock Cambridge University Press, 1996.

\bibitem{KKL2016}
S.~Kasjan, G.~Keller, and M.~Lema\'n{}czyk.
\newblock {Dynamics of $\cB$-free sets: a view through the window}.
\newblock {\em International Mathematics Research Notices}, 2019(9):2690--2734,
  2019.

\bibitem{Keller2017}
G.~Keller.
\newblock {Generalized heredity in $\mathcal B$-free systems}.
\newblock {\em To appear in Stochastics \& Dynamics}, 2019.

\bibitem{KR2015}
G.~Keller and C.~Richard.
\newblock Dynamics on the graph of the torus parametrisation.
\newblock {\em Ergod. Th. \& Dynam. Sys.}, 38(3):1048--1084, 2018.

\bibitem{Li-Vitanyi}
M.~Li and P.~Vit\'anyi.
\newblock {\em {An Introduction to Kolmogorov Complexity and its
  Applications}}.
\newblock Texts in Computer Science. Springer, 4th edition, 2019.

\bibitem{Moody2002}
R.~V. Moody.
\newblock Uniform distribution in model sets.
\newblock {\em Canad. Math. Bull.}, 45(1):123--130, 2002.

\bibitem{Tchebichef1852}
M.~Tchebichef.
\newblock {M{\'{e}}moire sur les nombres premiers}.
\newblock {\em Journal de Math{\'{e}}matiques Pures et Appliqu{\'{e}}es 1re
  S{\'{e}}rie}, 17:366--390, 1852.

\bibitem{Williams1984}
S.~Williams.
\newblock Toeplitz minimal flows which are not uniquely ergodic.
\newblock {\em Z.~Wahrscheinlichkeits\-theorie verw.~Gebiete}, 67:95--107,
  1984.

\end{thebibliography}
\end{document}